\documentclass{amsart}
\usepackage[shortlabels]{enumitem}
\usepackage{theoremref}
\usepackage{comment}
\usepackage{mathrsfs}

\newcommand{\1}{\mathbf{1}}
\newcommand{\f}{\frac}
\newcommand{\ind}[1]{\mathbf{1}{\{ #1 \}}}
\newcommand{\Var}{\mathrm{Var}}

\DeclareMathOperator{\Poi}{Poi}

\DeclareMathOperator{\Bin}{Bin}

\newtheorem{theorem}{Theorem}
\newtheorem{lemma}[theorem]{Lemma}
\newtheorem{prop}[theorem]{Proposition}

\theoremstyle{remark}
\newtheorem{remark}[theorem]{Remark}
\theoremstyle{definition}

\usepackage{mathtools}
\mathtoolsset{showonlyrefs}

\title{Two-type annihilating systems on the complete and star graph}
\author[Cristali]{Irina Cristali}
\email{icristali@statistics.uchicago.edu}

\author[Jiang]{Yufeng Jiang}
\email{yj49@uw.edu}

\author[Junge]{Matthew Junge}
\email{Matthew.Junge@baruch.cuny.edu}

\author[Kassem]{
Remy Kassem}
\email{remy.kassem@duke.edu}

\author[Sivakoff]{
David~Sivakoff}
\email{dsivakoff@stat.osu.edu}

\author[York]{Grayson York}
\email{grayson.york@duke.edu}

\thanks{The third author was partially supported by NSF-DMS grant 1855516.}

\begin{document}
	\maketitle

\begin{abstract}
%We study the number of steps needed to remove all particles in one- and two-type annihilating systems on the complete graph and the star graph.
 Red and blue particles are placed in equal proportion throughout either the complete or star graph and iteratively sampled to take simple random walk steps. Mutual annihilation occurs when particles with different colors meet. We compare the time it takes to extinguish every particle to the analogous time in the (simple to analyze) one-type setting. Additionally, we study the effect of asymmetric particle speeds.
 %These findings are in line with conjectures and results from physicists and mathematicians on infinite lattices.
%We show that two-type systems take more time than the one type system. Moreover, we consider the case of asymmetric speeds and show that the star graph has a different behavior in this regime. 
%This aligns with results from mathematicians and physicists for similar annihilating systems on infinite lattices.
\end{abstract}

\section{Introduction}

We introduce a discrete-time annihilating particle system and study the effects of multiple particle types and asymmetric speeds on the time to extinguish every particle. We consider such systems in two geometries: the {complete graph} on $2n$ vertices, $K_{2n}$, and the {star graph} with $2n$ leaves and a single non-leaf vertex, called the \emph{core}, $S_{2n}$. Initially, one particle is placed at every site of $K_{2n}$, or at every leaf of $S_{2n}$. 
In the \emph{one-type system}, at each step a particle is chosen uniformly at random and takes one step of a simple random walk. When any two particles meet, they mutually annihilate. 

In our \emph{two-type system}, half of the particles are colored blue and half are colored red. At each step, a blue particle is chosen uniformly at random with probability $p\in [1/2,1]$, or else a red particle is chosen uniformly at random, and the chosen particle takes a random walk step. When two particles with different colors meet, they mutually annihilate; particles of the same color do not interact. 
Note that the incremental movement of particles corresponds to the embedded jump chain from the analogous process with particles performing continuous time random walks. Increasing $p$ is equivalent to increasing the rate at which blue particles jump. 
%Assuming that each blue particle jumps according to a rate-1 Poisson point process and red according to a rate-$\lambda$ Poisson point process, then the probability a blue particle is next to be sampled is equal to $p=(1+\lambda)^{-1}.$
%Thus, a given choice of $p$ corresponds to $\lambda = p^{-1} -1.$ 
Accordingly, we call the case $p=1/2$ the \emph{symmetric speeds case} and $p>1/2$ the \emph{asymmetric speeds case}. 

The two-type system belongs to a family of processes that model systems with two compounds in which reactions neutralize both chemicals involved. These dynamics have been rigorously studied on infinite graphs, typically lattices. The main focus is on the asymptotic density of the different particle types. We initiate the study of such systems on finite graphs. Besides being natural for studying reactions with inherently limited space and material, the finite setting also introduces a new quantity: the time to neutralize all reactants.  By working with simple geometries---complete and star graphs---we reveal complicated underlying features of the dynamics.  For example, clustering of like compounds is a phenomenon that makes the two-type system more challenging to analyze than related one-type systems. Results for these dynamics on more ``realistic" finite graphs, such as tori and random networks, or for general topologies would be natural next steps. However, as with the infinite setting, rigorous results appear difficult to obtain.  More background and references are provided in Section \ref{sec:discussion}.
%We remark more on this in Section \ref{sec:questions}.

%Erd\H{o}s and Ney \cite{erdos1974} first proposed the mathematical study of annihilating systems and this spawned the first generation of work by Lootgieter \cite{lootgieter1977problemes}, Schwartz \cite{schwartz1978hitting}, and Arratia \cite{arratia1981limiting}. The introduction of two-type systems was first proposed by physicists \cite{chem1, chem2} and studied rigorously in a series of works \cite{BL2, BL3,BL4, BL5} by Bramson and Lebowitz. Despite recent interest in such systems with asymmetric speeds \cite{vicius_DLA, parking, tree}, many fundamental questions remain unanswered. See Section \ref{sec:discussion} for more details. %Our goal is to fill in our understanding of such systems on finite graphs.

 %See Section~\ref{results} for careful statements and discussion of our results.
	%	\item The $2$-star $S_2$ which is the star with two leaves. We will only consider the two-type system on this graph. 

\subsection{Results}\label{results}

Let $T^1(G)$ and $T^2_p(G)$ be the numbers of steps it takes in the one-type and two-type systems for every particle to be annihilated on the graph $G$. We begin with an informal summary of our results. It is straightforward and elementary to compute the distributions of $T^1(K_{2n})$ and $T^1(S_{2n})$ exactly.  This is done in~\thref{thm:1}, which we include for comparison with our quantitative bounds on $ET^2_p(G)$. In particular, we show that for $G= K_{2n}$ and $G=S_{2n}$ and for all $p\in [1/2,1]$, we have $ET^2_p(G)$ is asymptotically larger than $ET^1(G)$. How much larger depends of course on the particular graph and the value of $p$.
 For the complete graph we prove that $$2 n \log n \leq E T_p^2(K_{2n}) \leq 20 n(\log n)^2/\log\log n$$ for large  $n$, and in particular, $\liminf ET^2_p(K_{2n})/ET^1(K_{2n}) \ge 2$. Our strongest results are for the star graph. For $p=1/2$, we have $$c\sqrt{n} \le E T_{1/2}^2(S_{2n}) - E T^1(S_{2n}) \le C \sqrt{n} \log n$$ for large $n$. For $p\in(1/2,1)$, we have that $E T_{p}^2(S_{2n})/E T^1(S_{2n})$ is bounded away from $1$ and $\infty$ as $n\to\infty$ and diverges like $\log(1/(1-p))$ as $p\uparrow 1$; and for $p=1$, the ratio diverges like $2\log n$.

Throughout this article we let $X(p)$ denote a geometric random variable with distribution $P(X(p) = k) = (1-p)^{k-1}p$ for $k \geq 1$. We write $X \preceq Y$ to denote the usual notion of stochastic dominance $P(X \geq a) \leq P(Y \geq a)$ for all $a \geq 0$. Or, equivalently, that there is a coupling so that $X \leq Y$ almost surely. We say that $X \overset{d} = Y$ if $X$ and $Y$ have the same distribution. Our results make use of the standard asymptotic notation:
	\begin{itemize}
		\item $f=O(g)$ if $\limsup f/g < \infty$,
		\item  $f=\Omega(g)$ if $\liminf f/g > 0$, and
		\item $f= \Theta(g)$ if $f=O(g)$ and $g= O(f)$.
		\item We write $f \sim g$ if $\lim f/g = 1$. 
	%	\item An expression of the form $f(n) = C n + C' \sqrt n + \Omega(1)$ means that $f(n) \geq Cn + C' \sqrt n +C''$ for large enough $n$ and constants $C,C',C'' \geq 0$. 
	\end{itemize}

One can exactly characterize how long it takes to go from having $2i$ to $2(i-1)$ particles in the system in terms of a geometric random variable. Though elementary, this gives us a baseline for comparing to the two-type system. 

\begin{prop} \thlabel{thm:1} In both distributional equalities below the geometric random variables being summed are independent. \quad   \begin{enumerate}[label = (\roman*)]
	\item $T^1(K_{2n}) \overset{d}= \sum_{i=1}^n X(p_i)$ with $p_i = (2i-1)/2n $. In particular, $$E T^1(K_{2n})- (n \log n +\gamma n)  =  \Theta(1)$$
		where $\gamma = \lim (-\log n + \sum_1^n i^{-1})$ is the Euler-Mascheroni constant.
	\item $T^1(S_{2n}) 	\overset{d}= 2\sum_{i=1}^n X(q_i)$ with $q_i = 1 - (\frac{1}{2i})(\frac{2n - 2i +1}{2n})$. In particular, $$E T^1(S_{2n})- (2n + 2 \log n ) = \Theta(1).$$
\end{enumerate}
\end{prop}

Precisely analyzing the two-type system appears to be much more difficult. The issue on the complete graph is that, as the process evolves, like-particles tend to cluster at the same sites. The clustering should not be too extreme. Namely, when there are $\Omega(n)$ particles, red and blue should occupy $\Omega(n)$ distinct sites at all times, and when there are $o(n)$ particles red and blue should be nearly perfectly spread out. However, there is dependence between which particles are removed and the number of particles at each site. This appears to make it difficult to prove that red and blue particles stay sufficiently spread out.

While we do not completely overcome the issues mentioned above, we are able to confirm that the two-type system survives longer than the one-type system. 
%However, unlike the striking asymptotic difference between $p_t$ and $\rho_t$ from \eqref{eq:ARW} and \eqref{eq:dlas}, we believe that $ T^1(K_{2n})$ and $T^2(K_{2n})$ are on the same order, but with different leading constants.
 Below we prove that $E T_p^2(K_{2n}) \geq 2 ET^1(K_{2n}) (1-o(1))$. This result should not be all that surprising since the two-type system in some sense has at least twice as many ``safe" sites for particles to jump to among the occupied sites as the one-type setting. We also prove an upper bound that differs by a logarithmic factor.

\begin{theorem} \thlabel{thm:Kn} For all $p\in [1/2,1]$ it holds that 
\begin{align}
T_p^2(K_{2n}) \succeq \sum_{i=1}^n X(i/2n) \label{eq:KLB} 
\end{align}
with the $X(i/2n)$ independent.
Thus, $E T_p^2(K_{2n}) - 2n \log n   =  \Omega(1).$ Furthermore, the distributional inequality is an equality when $p=1$, so $E T_1^2(K_{2n}) - 2(n \log n+\gamma n)= \Theta(1)$. As for an upper bound, it holds for any fixed $p \in [1/2,1]$ that 
\begin{align}
E T_p^2(K_{2n})-  \frac{20 n (\log n)^2}{\log\log n} = O(1). \label{eq:KUB} 
\end{align}

\end{theorem}

 The proof of the lower bound uses a comparison to a process that has red and blue particles take up the maximal amount of space at each time step. Analogous to what occurs in \thref{thm:1}, we show that $T^2_p(K_{2n})$ stochastically dominates a sum of geometric random variables. The upper bound goes by showing that it is overwhelmingly unlikely for any site to host more than $C\log n/\log \log n$ particles in the first $n^3$ steps of the process. This gives a tractable way to lower bound the probability of a collision, but comes at the cost of the additional logarithmic factor.
 
We can say more about the two-type system on the star graph. The process in this setting has the same clustering issue at the leaves as what occurs globally on $K_{2n}$. Moreover, the number of like-particles grouped at the core introduces another hub for many like-particles to cluster. While, in principle, one could write down an explicit Markov chain for this process, to do this precisely would require keeping track of the number of particles at the core, as well as the number of red and blue leaves with $1,2,\hdots$ particles at them. 
   Analyzing this Markov chain exactly appears challenging since the state-space is of growing dimension, and there is significant dependence between the transition rate and the population size.
   
Despite these difficulties, we prove a fairly precise characterization for all $p \in [1/2,1]$. For symmetric speeds we show that the second order term is different for the two-type case. The proof reveals that this is caused by clustering at the core. Recall that \thref{thm:1} shows $ET^1(S_{2n})$ has a logarithmic second order term. We show that $E T_{1/2}^2(S_{2n})$ has a second order term on an order between $\sqrt n$ and $\sqrt n \log n$. This demonstrates the effect of clustering at the core and, along with \thref{thm:1} (ii), also implies that $E T_{1/2}^2(S_{2n}) - E T^1(S_{2n}) = \Omega(\sqrt n).$

\begin{theorem} \thlabel{thm:2} It holds that
\begin{enumerate}[label = (\roman*)]
	\item	$E T_{1/2}^2(S_{2n}) -( 2n +  C \sqrt{n} ) = \Omega(1)$ for any $C < (32 \pi)^{-1/2}$, and
	\item $E T_{1/2}^2(S_{2n}) - (2n+ c \sqrt n \log n)  =  O(1)$ for some $c>0$.
\end{enumerate}
\end{theorem}

 The starting point for the lower bound is a ``master formula'' in \thref{lem:A_t} that equates the number of remaining particles to what has occurred up to that point at the core. We use this to make estimates on the number of particles in the system at time $2n$. This relies on a coupling to the simple random walk which tracks the discrepancy between the number of times red and blue have been sampled.
The upper bound again uses the identity in \thref{lem:A_t}, but this time couples to a different random walk to estimate the number of particles clustered at the core as the process evolves.  The argument concludes by bounding the probability a particle is sampled at the core.
%This may be slightly suboptimal and explains the appearance of the logarithmic factor in the second order term. 

 For asymmetric speeds, we focus on the leading order coefficient and provide universal upper and lower bounds. The lower bound implies that the asymmetric case has a strictly larger leading coefficient than the symmetric case.
 
\begin{theorem} \thlabel{thm:p}
	Fix $p \in (1/2,1)$. 
	 It holds  for all $n$ that 
\begin{align}\left(2 + \f{2p-1}2 \right) n -1 \leq E T_p^2(S_{2n}) \leq \f{2}{(1-p)}n.\label{eq:ULB}	
\end{align}
\end{theorem}

The lower bound is proven in a similar manner as \thref{thm:2} (i), and the upper bound follows from the observation in \thref{lem:remy_bound} that from any configuration, after two steps, the probability of a collision is at least $1-p$. So $T^2_p(S_{2n})$ is stochastically dominated by a sum of independent geometric random variables.
While these bounds hold for all $p$, they become rather far from the truth for $p$ near $1$. The following theorem addresses what happens in this regime. We provide matching order upper and lower bounds for the rate the leading constant tends to infinity.  %\HOX{I think this last sentence needs to be rephrased or split up into two sentences. -RK}

%\HOX{need to figure out order of limits for $p$ and $n$. -MJ}
\begin{theorem} \thlabel{thm:asymptotic}
Given $c<4$ there exists a value $p_* <1$ (which depends on $c$) such that for any fixed $p \in (p_*,1)$ it holds that
\begin{align}
E T_p^2(S_{2n}) - c \log \left( \f 1 {1-p} \right) n = \Omega(1). \label{eq:ALB}
\end{align}
And, given $C>12$ there exists $p^*<1$ (which depends on $C$) such that for any fixed $p \in (p^*,1)$ it holds that
\begin{align}
 E T_p^2(S_{2n}) - C \log \left( \f 1 {1-p} \right) n = O(1). \label{eq:AUB}
\end{align}
\end{theorem}

These results use stochastic lower and upper bounds that relate $T^2_p(S_{2n})$ to a coupon collector process. This connection was not so obvious to make, and requires technical estimates on the time to collect a random subset of the coupons, as well as on the number of coupons collected after a random amount of time.

Finally, we handle the case $p=1$, which corresponds to the setting in \cite{parking}. This setting is  tractable because the now immobile red particles cannot cluster.
\begin{theorem}\thlabel{prop:p=1}
It holds that $$T_1^2(S_{2n}) \overset{d}=  2\sum_{i=1}^n X(p_i)$$ with the $X(p_i)$ independent and $p_i = i/2n$. In particular, 
	$$E T_1^2(S_{2n}) -( 4 n \log n + 4 \gamma n) =   \Theta(1)$$
	with $\gamma$ the Euler-Mascheroni constant.
\end{theorem}
%The proof of \thref{prop:p=1} uses the identity in \thref{lem:A_t} and decomposes annihilation events based on the behavior of particles sampled to move from the core.

\subsection{Background} \label{sec:discussion}

The study of annihilating particle systems dates back to the work of Erd\H{o}s and Ney \cite{erdos1974}. They considered a system of continuous time random walks started at each nonzero integer in which collisions cause both particles to annihilate and disappear from the process. In particular, they asked if the origin was visited infinitely often, and, more precisely, they studied the asymptotic decay of $p_t$, the probability the origin is occupied at time $t$. 

The question of whether or not the origin is visited infinitely often was answered in the affirmative by Lootgieter in \cite{lootgieter1977problemes} in discrete time and by Schwartz in \cite{schwartz1978hitting} in continuous time. Later, Arratia in \cite{arratia1981limiting, arratia} generalized the process to higher dimensions and more general initial configurations. One of his main findings was that
\begin{equation}
p_t \sim \begin{cases}  1/(2 \sqrt{\pi t}), & d =1 \\
\log t /( 2 \pi t), & d=2 \\ 
1/(2 \gamma_d t), &d\geq 3 \end{cases}\label{eq:ARW}
\end{equation}
where $\gamma_d$ is the probability the simple random walk never returns to its starting position in $\mathbb Z^d$. Due to a parity relation observed by Arratia, $p_t$ decays exactly twice as fast as what Bramson and Griffeath in \cite{bramson1980asymptotics} proved occurs for \emph{coalescing random walk}. This is  the system in which particles coalesce rather than annihilate upon colliding. The main proof technique in these systems is to analyze a dual process known as the voter model.

Two-type annihilating particle systems first garnered interest in the chemistry and physics literature~\cite{chem1, chem2,AB3,AB1}. Initially particles are assigned to be either of type $A$ or $B$, and only collisions between different particle types result in annihilation. Unlike the one-type annihilating and coalescing systems, the two-type system has no known tractable dual process. Ovchinnikov and Zeldovich and later Toussaint and Wilczek predicted that in low dimensions the density of particles at the origin of $\mathbb Z^d$ is asymptotically much larger than in the one-type system \cite{chem1, chem2} due to local clustering of like particles.
%The different behavior was thought to occur because of the tendency of like particles to cluster, thus prolonging survival.

In a definitive series of papers, Bramson and Lebowitz~\cite{BL2, BL3,BL4, BL5} proved this (and more) for the two-type system on $\mathbb{Z}^d$, where initially
%according to independent Poisson random fields with intensities $\mu_A,\mu_B>0$. This means that 
each site has a $\Poi(\mu_A)$-distributed number of $A$ particles and a  $\Poi(\mu_B)$-distributed number of $B$ particles. At time $0$, pairs of $A$ and $B$ particles at the same site instantly annihilate. Particles then perform continuous time simple random walks at rates $\lambda_A$ and $\lambda_B$, and annihilate when they meet a particle of opposite type.  Since multiple particles can occupy a given site, the main quantity of interest is the expected number of particles at the origin at time $t$, which we denote by $\rho_t$.
In the critical case, with particle types in balance ($\mu_A = \mu_B>0$) and symmetric speeds ($\lambda_A = \lambda_B>0$), Bramson and Lebowitz~\cite{BL4} proved that
\begin{align}\rho_t \approx \begin{cases}t^{-d/4}, & d \leq 3 \\ t^{-1}, & d \geq 4 \end{cases}\label{eq:dlas}.
\end{align}
Here $f \approx g$ if $0 < \liminf f/g \leq \limsup f/g < \infty.$
Note that, in low dimension, this is asymptotically much larger than the formula for $p_t$ at \eqref{eq:ARW}.
% We are interested in exploring the effect of clustering and asymmetric speeds for two-type annihilating systems in the finite setting. 

%Variants of these systems have seen a resurgence in interest on more general graphs with $A$- and $B$-particles diffusing at different rates \cite{parking, vicius_DLA}.

There has been recent interest in extending the results of Bramson and Lebowitz to asymmetric speeds. On lattices, physicists predicted that the asymptotic order of $\rho_t$ does not change as the speeds are varied \cite{AB3, AB1}. Cabezas, Rolla, and Sidoravicius in \cite{vicius_DLA} considered the asymmetric speed case on a class of infinite transitive graphs and proved a universal lower bound $\rho_t = \Omega(t^{-1})$, and that the root is visited infinitely often when particle types are initially in balance. In a different work \cite{Cabezas2014}, Cabezas, Rolla, and Sidoravicius considered the case that red particles move and blue particles are stationary. They proved that there is a phase transition between transience and recurrence when the different particle types are in balance on a broad class of transitive graphs. An Abelian property ensures that the results hold in either discrete or continuous time. More recently Johnson, Junge, Lyu, and Sivakoff proved new upper and lower bounds for the particle density in two-type annihilating systems on lattices and bi-directed regular trees \cite{DLAS}. Bahl, Barnet, Johnson and Junge further explored how the volatility of the distributions of the initial particle counts impacts the total occupation time of the root \cite{bahl2021diffusion}. 

Damron, Gravner, Junge, Lyu, and Sivakoff considered a similar problem as \cite{Cabezas2014} in discrete time and proved transience/recurrence results along with more quantitative estimates on the number of visits to the origin when the particle densities are initially out of balance \cite{parking}. Very recently, Przykucki,  Roberts, and Scott proved quantitative results in discrete time with $B$-particles stationary on the integers \cite{przykucki2019parking}. A slightly different, but related process was studied by Goldschmidt and Przykucki on Galton-Watson trees \cite{tree}. The papers \cite{parking, przykucki2019parking, tree} refer to the annihilating system as \emph{parking} since they view $A$-particles as cars and $B$-particles as parking spots. 
Parking was introduced over fifty years ago in \cite{KW66} and has attracted interest ever since. See \cite{uniform_parking} for an overview. 
%To our knowledge, the problem of pinning down the order of $\rho_t$ for asymmetric speed two-type annihilating systems on any infinite graph is an open problem.

As for the finite setting, Cooper, Frieze, and Radzik studied similar quantities as us on random regular graphs \cite{cooper2009multiple,CFR09}. They considered an ``explosive" particle system with the same dynamics as our one-type system, but with the modification that all particles move simultaneously. They proved that the time it takes to remove all particles is $O(n \log n)$ when there are sufficiently few particles initially. Additionally, the authors considered a two-type ``predator-prey" dynamics, in which predators remove prey on contact, but predators persist, and they studied the expected time to remove all prey.
These quantities are closely related to the \emph{coalescence time}. This is the number of steps needed to reach a single particle when particles coalesce, rather than annihilate, upon colliding. There have been recent results for how this behaves on general finite graphs \cite{cooper2013coalescing, kanade2019coalescence}, as well as a result from Cox concerning the coalescence time on the torus \cite{cox1989coalescing}. Note that \cite{cox1989coalescing,cooper2013coalescing, kanade2019coalescence} only considered one-type systems. To the best of our knowledge, the quantity $T^2_p(G)$ for two-type systems has not been studied on any finite graph.

%Our main interest in annihilation dynamics comes from the physics particle system perspective. In particular, we wish to demonstrate that clustering prolongs the survival of particles in the two-type systems. This appears to be a difficult question in general. Indeed, there is no coupling we know of that directly shows that one-type systems extinguish faster than two-type systems on an arbitrary graph. 
%To get started, we consider perhaps the tamest non-trivial setting: complete graphs and stars with incremental movement. 

\subsection{Further questions} \label{sec:questions}

It would be interesting to find the correct leading order coefficient for $E T_p^2(K_{2n})$ and for $E T^2_p(S_{2n})$. Note that currently we do not have a proof that $E T_p^2(K_{2n}) = O(n \log n)$. For the star graph, we conjecture that our asymptotic lower bound in \thref{thm:asymptotic} is sharp so that, for large enough $p$, 
$$E T_p^2(S_{2n}) \sim   4\log\left( \f 1 {1-p}\right) n.$$
This is the answer one gets for the simplified model in which one assumes that the core is always occupied by blue particles and that every red step results in a collision. While the connection is difficult to make rigorous, it seems to be a reasonable approximation for large $p$. 
%Setting $p_n = 1- 1/n$ and letting $n \to \infty$, this would be consistent with the fact that $ ET_1^2(S_{2n}) \sim 4n \log n$ as proven in \thref{prop:p=1}. 
%\HOX{Why is this the correct heuristic? Why not take $p_n = 1-n^{-1/2}$, say, and get $8$ for the constant? -DS}
%Getting the matching upper bound may be tractable, but it also difficult because clustering of red particles is a 
%Also, it would be nice to show that 
%$$\sup_{p <1}  \left(\sup_n \f{E T_p^2(S_{2n})}{n} \right) < \infty.$$
%Our current bound on the leading coefficient diverges, but simulations suggest it is bounded.
%This is significantly slower than our upper bound, and faster than our lower bound which does not diverge with $p$. Setting $p = 1 - (1/n)$ and taking a limit, this conjecture aligns with the $n \log n$ behavior we see at $p=1$.
%

 We also would like to know the exact second order term for $E T_{1/2}(S_{2n})$. This is a more delicate question, but it would be interesting to decide if the logarithmic factor is needed, and if so, what is causing its appearance. We discuss this a bit more at \thref{rem:G}.
Another future direction is to understand two-type annihilating systems on other finite graphs, such as Erd\H{o}s-R\'enyi graphs, tori, and trees.

\subsection{Organization}

In Section \ref{sec:1} we analyze the one-type system and prove \thref{thm:1}. In Section \ref{sec:2} we prove our lower and upper bounds for the two-type system on the complete graph from \thref{thm:Kn}. In Section \ref{sec:2S} we analyze the two-type system with symmetric speeds on the star graph by proving the upper and lower bounds in \thref{thm:2}. Section \ref{sec:asymmetric} houses the proofs of Theorems \ref{thm:p}, \ref{thm:asymptotic} and \ref{prop:p=1} for asymmetric speed two-type systems on the star graph.

\section{One-type systems} \label{sec:1}
One-type systems are fairly straightforward to precisely describe because at most one particle can occupy each site. Combining this feature with the simple geometry of the complete and star graphs makes it so the time to annihilate every particle decomposes as a sum of independent geometric random variables.

\begin{proof}[Proof of \thref{thm:1}]\quad 
We start with a general decomposition then explain how to prove (i) and (ii). Let $\tau_i$ be the first time there are $2i$ particles in the system for $0 \leq i \leq n$. Notice that $\tau_0 = T^1(G)$ and $\tau_n = 0$ so that
		\begin{align}T^1(G) = \tau_0 - \tau_n = \sum_{i=1}^n \tau_{i-1} - \tau_{i}.\label{eq:decompose}	
		\end{align}

%\HOX{ Should be $X(2i-1/2n)$ since this is the one-type system, and the explanation should be changed. -RK}
To prove (i), notice that on $K_{2n}$ we have $\tau_{i-1} - \tau_{i} \overset{d}= X((2i-1)/2n)$. This is true because, when there are $2i$ particles in the system, when a particle is selected there are $2i-1$ out of $2n$ sites with another particle that it could move to and cause a collision.

To prove (ii), for the star graph we claim that $\tau_{i-1} - \tau_{i} \overset{d} = 2 X(q_i)$ 	
	with 
	\begin{align}
	q_i = \f{(2i-1)(2n+1)}{4ni} \label{eq:qi}
	\end{align}
	and the $X(q_i)$ independent.
	Assuming \eqref{eq:qi} holds and applying this to \eqref{eq:decompose}, we then have
  \begin{align*}
    ET^1(S_{2n}) = 2 \sum_{i=1}^n \f1 {q_i} &= 2\sum_{i = 1}^n \frac{4ni}{(2i-1)(2n+1)} \\
    & = \frac{8n}{2n+1} \sum_{i = 1}^n \f 12\left(1 + \f1 { 2i -1} \right)  \\
%    & = \frac{8n}{4n+2} (n + \f 12 \log (2n - 1)) \\
%    &= \f{8 n^2}{4n+1} + \f{8n}{8n + 4} \log(2n -1)\\
    & = 2n + \log n + \log 2 + o(1).
  \end{align*}

It remains to justify that $\tau_{i-1} - \tau_{i} \overset d = 2 X(q_i)$ with $q_i$ as in \eqref{eq:qi}. First notice that at time $\tau_i$ there is no particle at the core, and $2i$ particles at the leaves.  This is because the only way an annihilation can occur on the star is by either (a) a particle moving from a leaf to the core, or (b) a particle moving from the core to a leaf. Since only one particle can occupy a given site, both (a) and (b) result in no particles at the core when annihilation occurs. Now, the next step from this configuration will necessarily be a particle from a leaf moving to the core. This particle is destroyed either if it is again selected and then moved to one of the $2i-1$ occupied leaves, or if a particle at a leaf is next selected; these occur with total probability
$$\f{1}{2i}\f{2i-1}{2n}+\f{2i-1}{2i} = q_i.$$
%Simplifying shows that the above is equal to $q_i$.
To obtain a renewal, notice that if the particle at the core is not destroyed in this second step, then it must move back to an unoccupied  leaf. On this event, we have no particle at the core, and $2i$ particles at the leaves, which was the configuration at time $\tau_i$. The next two steps once again result in an annihilation with probability $q_i$, so we have $\tau_{i-1} - \tau_{i} \overset{d} = 2 X(q_i).$
%
%	 Let $q_i$ be the probability that the $i^{th}$ pair of particles annihilate. We first calculate the probability that the $i^{th}$ pair do not annihilate. This situation can only occur if we first randomly pick one particle from the $2i$ particles, move it to the center, and then move it to one of the $2n - 2i +1$ open spots. Then, the probability of not annihilating is $= (\frac{1}{2i})(\frac{2n - 2i +1}{2n})$. Thus, $q_i = 1 - (\frac{1}{2i})(\frac{2n - 2i +1}{2n})$.
%
%  For each annihilation, 2 steps are needed: move to the center and move to another particle. Therefore, the steps to annihilation is 2 times the sum of geometric random variables. This gives $T_n^1(S_{2n}) 	\overset{d}= 2\sum_{i=1}^n X(q_i)$.
%
%  Then,
%  $$
%  \begin{aligned}
%    ET_n^1(S_{2n}) & = 2 \sum_{i = 1}^n \frac{4ni}{(2i-1)(2i+1)} \\
%    & = \frac{8n}{2n+1} \sum_{i = 1}^n \frac{i}{2i-1} \\
%    & = \frac{8n}{4n+2} (n + \sum_{i = 1}^n \frac{1}{2i-1}) \\
%    & \approx 2n + O(logn)
%  \end{aligned}
%  $$
\end{proof}

\section{Two-type systems on the complete graph} \label{sec:2}

 Let $A_t$ be the total number of particles remaining in the system after $t$ steps. The fact that collisions occur in pairs ensures that $0\leq A_t\leq 2n$ is even. It is also convenient to let $R_t$ and $B_t$ be the total number of sites with at least one red or blue particle, respectively. We start by giving the proof of the lower bound at \eqref{eq:KLB}, and then give the proof of the upper bound at \eqref{eq:KUB}.

\begin{proof}[Proof of \thref{thm:Kn} equation \eqref{eq:KLB}]
We start by describing the transition probabilities for $A_t$ conditional on the number of sites occupied by red and blue particles. Letting $0\le r,b\le i\le n$, we have
\begin{align}
P(A_{t+1} &= A_{t} -2 \mid R_t =r, B_t = b, A_t =2i)= p \Big(\frac{r}{2n}\Big) + (1-p) \Big(\frac{b}{2n}\Big).
\end{align}	
Otherwise $A_{t+1} = A_t$. Notice that the above equality does not depend on the value of $A_t$. When there are $2i$ particles remaining, the probability of an annihilation occurring on a given step is thus bounded by probability of an annihilation when $R_t=i=B_t$. This is equal to $p_i =  i/ {2n}.$
%P(A_{t+1} &= A_{t}  \mid R_t =r, B_t = b, A_t =a)= 1-  \frac{r+b}{4n}
%\end{align}
Using $p_i$ as a bound on the probability of an annihilation is equivalent to comparing to a process that always has red and blue particles occupying the maximal number of sites. For the case $p=1$, the value $p_i$ is the actual transition rate, because red particles do not move.

It follows that the number of steps to transition from $2i$ to $2(i-1)$ particles is stochastically larger than $X(p_i)$. Decomposing $T^2_p(K_{2n})$ into the time it takes to go from $2n$ to $2(n-1)$ to $2(n-2)$, and so on, we have 
$$T^2_p(K_{2n}) \succeq \sum_{1}^{n} X(p_i).$$ 
The well-known asymptotic behavior of  the harmonic series ensures that
\begin{align*}
E T^2_p(K_{2n}) \geq E \sum_{1}^{n} X(p_i)  = \sum_{1}^{n} \frac{2n}{i} = 2 n \log n + 2\gamma n  + O(n^{-1}).
\end{align*}
\end{proof}

\begin{proof}[Proof of \thref{thm:Kn} equation \eqref{eq:KUB}]
	We first show that with high probability no site has more than $m:= 6\log n / \log\log n $ particles through time $n^3$. Let $Z_t$ be the number of blue particles at vertex $1$ at time $t$. Note that $Z_t$ cannot jump by more than $1$ at any step, so we can dominate it by a birth and death chain. For $k\ge 0$,
\begin{align}
P(Z_{t+1} = k+1\ |\ Z_t = k, A_t = 2i) \le \frac{p}{2n}, \label{eq:p1}
\end{align}
since a blue particle must move to $1$. To decrease the number of blue particles at $1$, it suffices to choose a blue particle at $1$ and move it somewhere else, so for $0<k \le i$ and $n\ge 2$,
\begin{align}
P(Z_{t+1} = k-1\ |\ Z_t = k, A_t=2i) \ge p\cdot \frac{k}{i}\cdot \left(1-\frac{1}{2n}\right) > \frac{pk}{2n}. \label{eq:p2}
\end{align}

The ratio between the last two probabilities is at least $$(pk/2n)/(p/2n) =k,$$ independent of $i$. Therefore, independent of $(A_t)$, we have that $Z_t$ is dominated by a birth and death chain that is $k$ times as likely to move left as right when it is at $k$. Thus, whenever $Z_t=1$, the probability that it hits $m$ before hitting $0$ is at most $1/(m-1)!$ (see, for instance, Example 5.3.9 in \cite{durrett}).
%\HOX{add PTE v.5 to references} 
Using the bound $m! \ge (m/e)^m$, the probability that $(Z_t)$ reaches $m$ by time $n^3$ is at most 
\begin{equation}\label{particles at 1 bd}
\frac{n^3}{(m-1)!} \le n^3 \frac{e^m}{m^{m-1}}= n^3 \exp\left[-m \log m + m + \log m\right]
\end{equation}
independent of $(A_t)$. We obtain the same bound for the probability that the number of red particles at $1$ hits $m$ by time $n^3$ by repeating the same argument as above, replacing the $p$-factor with $1-p$ in \eqref{eq:p1} and \eqref{eq:p2}.

Define the events
$$
G_t = \{\text{every site has at most $m$ particles through time } t\},
$$
so the union bound and~\eqref{particles at 1 bd} imply that for all $t\le n^3$ and for all sufficiently large $n$, we have 
\begin{equation}
\begin{aligned}
P(G_t^c \mid A_t = 2i) &\le 2n^4  \exp\left[-m \log m + m + \log m\right] \\
%& = 2n^4 \exp\left[-6\frac{\log n}{\log\log n} (\log \log n + \log 6  - \log\log\log n) + 6\frac{\log n}{\log\log n} + \log \left(6\frac{\log n}{\log\log n}\right)\right]\\
%&\le n^4 \exp\left[-5 \log n\right]\\
& \leq n^{-1}.
\end{aligned}
\end{equation}
%\HOX{After checking, I think we can remove the 2nd and 3rd lines in the last display - DS}
On the event $\{A_t = 2i\}\cap G_t$, there are at least $i/m$ sites that contain at least one blue particle, and at least $i/m$ sites that contain at least one red particle. A collision occurs if a particle is selected and moves to one such site containing the opposite type. For $1\le i\le n$ and $t\le n^3$, and assuming $n$ is sufficiently large, we have
\begin{equation*}
\begin{aligned}
P(A_{t+1} = 2i-2 \mid  A_t = 2i) &\ge P(A_{t+1} = 2i-2 \mid \{A_t = 2i\}\cap G_t) \cdot P(G_t \mid  A_t = 2i) \\
&\ge  \frac{i/m}{2n} \cdot (1-1/n)\\
& \ge \frac{i}{3nm} =: r_i.
\end{aligned}
\end{equation*}
Therefore, for $t\le n^3$, we have
$$
P(T_p^2(K_{2n}) \ge t) \le P\left( \sum_{i=1}^n X(r_i) \ge t\right).
$$
For $t> n^3$, letting $S_n = X_1 +\cdots + X_n$ be the sum of $n$ i.i.d.~Geometric($1/2n$) random variables, we have the trivial upper bound,
$$
P(T_p^2(K_{2n}) \ge t) \le P(S_n\ge t) \le n P(X_1 \ge \lfloor t/n\rfloor) = n(1-1/2n)^{\lfloor t/n\rfloor-1} \le e^{-t/(4n^2)}
$$
for large $n$.
Summing over $t$ and using the last two inequalities for $t\le n^3$ and $t>n^3$, respectively, we have
\begin{equation}\label{K_2n upper bd}
ET_p^2(K_{2n}) \le E\sum_{i=1}^n X(r_i) + \frac{e^{-n/4}}{1-e^{-1/(4n^2)}}.
\end{equation}
The first term is equal to
$$
\sum_{i=1}^n \frac1{r_i} = 3nm \sum_{i=1}^n \frac1i \le 3 nm(\log n + 1).
$$
The second term in~\eqref{K_2n upper bd} is bounded by $8 n^2 e^{-n/4}$ for large $n$, so tends to $0$ as $n\to\infty$. We have proved that for large $n$,
$$
ET_p^2(K_{2n}) \le 3 nm(\log n + 2) \le  \frac{20n(\log n)^2}{\log\log n}.
$$	
\end{proof}

\section{The star graph with symmetric speeds} \label{sec:2S}
We start by fixing some notation. Again let $A_t$ be the total number of particles in the system after $t$ steps. Let $C_t$ be the number of particles that are at the core after $t$ steps. Additionally, let $M_t$ be the number of times up to time $t$ that a particle at the core is sampled to move, but is not annihilated after taking a step.
Let $Z_t = 1$ if blue is sampled at time $t$, and $-1$ if red is sampled (note that $(Z_t)_{t\ge1}$ is an i.i.d.~sequence, defined even for $t\ge T^2_p(S_{2n})$ after there are no particles remaining). Define the quantities 
\begin{align}W_t = \textstyle \sum_{s=1}^t Z_s \text{ and } D_t = |W_t|, \label{eq:DW}
\end{align}
so that $D_t$ has the same law as the displacement of a $p$-biased random walk. When we write $D_{2n}$ it is implicit that this is the value of $D_t$ at $t=2n$ for the process on $S_{2n}$.

The quantities $A_t$, $C_t$ and $M_t$ are related by the following identity, which will be useful for proving both lower and upper bounds on $ET^2_p(S_{2n})$.

\begin{lemma} \thlabel{lem:A_t}
For all $p \in [1/2,1]$ and $t \leq T_p^2(S_{2n})$ we have
\begin{align}A_{t} = 2n - t + C_t + 2M_{t}.\label{eq:start}
\end{align}
Moreover, for $t\le 2n$, we have
$$A_{t} \ge 2n - t,$$
and consequently $T_p^2(S_{2n}) \ge 2n$.
\end{lemma}

\begin{proof}
Clearly the formula holds for $t=0$. We proceed inductively from here. Suppose that \eqref{eq:start} holds through step $t$. At step $t+1$, we sample a particle to move; call this particle $x$. If the step taken by $x$ results in a collision, then by \eqref{eq:start},
\begin{align}
A_{t+1} = A_{t} -2 = 2n - t -2 + C_t + 2M_t.\label{eq:induction}
\end{align}
Either the collision happens at the core, or $x$ moves from the core to a leaf. In either scenario we have $C_{t+1} = C_t -1$ and $M_{t+1} = M_t$, so
$$A_{t+1} = 2n - t -2 + (C_{t+1} +1) + 2M_{t+1} = 2n - (t+1) + C_{t+1} + 2M_{t+1}.$$
This is the desired statement at time $t+1$.

Now, suppose that the step taken by $x$ does not result in a collision. If $x$ moves from a leaf to the core, then $C_{t+1} = C_t + 1$, and if $x$ moves from the core to a leaf, then $M_{t+1} = M_t + 1$ and $C_{t+1} = C_t -1$. In the first case, we have
$$A_{t+1} = A_t = 2n - t + C_t + 2M_t = 2n - t + C_{t+1}-1 + 2M_{t+1}.$$ 
In the second case we have
$$A_{t+1} = A_t = 2n-t + (C_{t+1}+1) + 2(M_{t+1} -1).$$
Simplifying either case gives $A_{t+1} =  2n - (t +1) + C_{t+1} + 2M_{t+1}$, as desired.

The second and third statements follow from~\eqref{eq:start} by observing that $C_t\ge 0$ and $M_t\ge0$ for $t=0,1,\ldots,2n-1$.
\end{proof}

%We also record the following simple observation, which can be derived from~\thref{lem:A_t} by noting that $C_t\ge 0$ and $M_t\ge0$ for $t=1,\ldots,2n-1$.
%\begin{lemma}\thlabel{lem:min steps}
%For all $p\in [1/2,1]$ it holds that $T_p^2(S_{2n}) \ge 2n$.
%\end{lemma}

\subsection{A lower bound for the symmetric case}
The starting point for our lower bound is a simple observation that relates $T_p^2(S_{2n})$ to the process stopped at a given time.

\begin{lemma}\thlabel{lem:TLB}
For all $p \in [1/2,1]$ and $t \leq T_p^2(S_{2n})$ it holds that	$$T^2_p(S_{2n}) \geq t + A_t/2.$$  
\end{lemma}

\begin{proof}
At most two particles can be removed from the sytem  at each step. Thus, if there are $A_t$ particles at time $t$, then it deterministically takes at least $A_t/2$ more time steps to remove them all. 
%Note that $A_t$ is always even so $A_t/2$ is an integer, and that this bound only works up to time $T_p^2(S_{2n})$, because beyond that time $A_t$ is always zero. 
\end{proof}

We can further bound $M_{2n}$ in terms of $D_{2n}$.

\begin{lemma} \thlabel{lem:M_t}
For any $p \in [1/2,1]$, we have $E M_{2n} \geq (1/8)E  D_{2n}   -1$. 
\end{lemma}

\begin{proof}
Without loss of generality, suppose that blue is sampled $n + D_{2n}/2$ times through time $2n$. Let $\alpha$ be the number of times through time $2n$ that a blue particle moves from the core to a leaf. By \thref{lem:A_t} we have $T_p^2(S_{2n}) \ge 2n$, so $(n+D_{2n}/2) - \alpha$ is the number of times that a blue particle moves from a leaf to the core. Since we have only $n$ blue particles initially, we must have
$$
[(n+D_{2n}/2) - \alpha] - \alpha \le n,
$$
so $\alpha \ge D_{2n}/4$. Since red can occupy at most half of the leaves, each core selection of blue has at least a $1/2$ chance of increasing the count of $M_{2n}$. In particular, $M_{2n}$ stochastically dominates a binomial thinning of $D_{2n}/4$ with success probability $1/2$. The same holds when red is sampled $D_{2n}$ times more than blue. Using the bound $\lfloor E D_{2n}/8 \rfloor  \geq ED_{2n}/8 - 1$ gives the claimed inequality. 
\end{proof}

%\begin{proof}
%Label each red and blue particle and denote those with label $i$ by $r_i$ and $b_i$, respectively. After running the process to time $T_n$, let $\ell(i)$ be the label of the red particle that eventually annihilates with $b_i$. Furthermore, let $\eta_t(i)$ be the number of steps taken by $b_i$ plus the number taken by $r_{\ell(i)}$ up until time $t$. 
%%Let $\eta(i) = \eta_{T_n}(i)$. This decomposition gives, for example,
%%$$T_n = \sum_{i=1}^{n} \eta(i).$$
%%Additionally, if we conside
%Let $M_t(i)$ be the number of times that either $b_i$ or $r_{\ell(i)}$ is selected while at the center and survives the next step. Since particles alternate between occupying the core and occupying leaves, we have 
%$$\eta_i = 2 + 2M_{T_n}(i).$$
%
%
%Let $C(i)$ be the number of times $b_i$ or $r_{\ell(i)}$ is selected to move while at the core, and is not annh
%
%\end{proof}

%Let $M= M_{T_n}$ be the total number of times this occurs. Setting $t = T_n$ in \thref{lem:AM} and using the fact that $A_{T_n} = 0 = C_{T_n}$ gives the useful characterization 
%\begin{align}
%T = 2n + 2M	\label{eq:TM}.
%\end{align}
%
%

It is a well-known estimate that $E D_t$ grows like $\sqrt t$ when $p=1/2$. We give a combinatorial proof of this fact below.

\begin{lemma}\thlabel{lem:D_t}
Suppose $p=1/2$. It holds that $E D_t \leq \sqrt t$ for all $t\geq 0$. Moreover, as $n \to \infty$ it holds that $E D_{2n} \sim \sqrt{2n / \pi}.$
\end{lemma}

\begin{proof}

Recalling the definition at \eqref{eq:DW}, it is a standard exercise to show that $W_t^2 - t$ is a martingale, and thus $EW_t^2 = t$. We then have
$$E D_t \leq \sqrt {E W_t^2 } = \sqrt t.$$

Next we prove the asymptotic claim. Observe that for every integer $x> 0$, $E[D_{n+1}\mid D_n=x]=x$, while $E[D_{n+1}\mid D_n=0]=1$. We then have $$E D_{n+1}=E D_n +P(D_n=0).$$ Using the parity observation that $D_{2n+1} \neq 0$, and that $E D_1 =1$, gives the equation 
%\HOX{I think the last term should have a "1+". -RK}
$$E D_{2n} = 1+\sum_{k=1}^{n-1}P(D_{2k}=0)=1 +\sum_{k=1}^{n-1}2^{-2k}{2k\choose k}.$$ Stirling's approximation then yields $$2^{-2k}{2k\choose k}\sim\frac1{\sqrt{\pi k}}.$$
Integrating $1 / \sqrt{\pi k}$ from $1$ to $n$ gives the claimed asymptotic formula for $E D_{2n}$. 
%The result for $E D_t$ follows by setting $C = \limsup E D_t / \sqrt t$ which is finite by the same Stirling's approximation. 
\end{proof}

\begin{proof}[Proof of \thref{thm:2} (i)]
Evaluating the formula in \thref{lem:A_t} at $t= 2n$ and ignoring the $C_{2n}$ term gives $A_{2n} \geq  2M_{2n}$. \thref{lem:M_t} then tells us that $$E M_{2n} \geq (1/8) E D_{2n} -1,$$ which by \thref{lem:D_t} is $(1/8)\sqrt{2 n / \pi} + o(\sqrt n)$. Thus, for any $C'< (32 \pi)^{-1/2}$ we have
	$$E A_{2n} - C' \sqrt n =  \Omega(1).$$
The result then follows by applying the above bound on $E A_{2n}$ to the inequality $E T^2_{1/2}(S_{2n}) \geq 2n + E A_{2n}/2 $ implied by \thref{lem:A_t,lem:TLB}.
\end{proof}

\subsection{An upper bound  for the symmetric case}

We start with a simple observation that provides a stochastic upper bound on $T^2_p(S_{2n})$ by stopping the process at a given time and then using a worst-case upper bound related to the number of particles still in the system at that time. 
\begin{lemma} \thlabel{lem:remy_bound}
For all $p \in [1/2,1)$ we have
$$T^2_{p}(S_{2n}) \preceq t + 2\sum_{i=1}^{A_{t}/2}X_i(1-p),$$ 
where $(X_i(1-p) : i\ge 1)$ are i.i.d.~Geometric$(1-p)$, and are independent of $A_t$; the sum on the right is  $0$ when $A_t=0$.
 \end{lemma}
\begin{proof}
If the core is occupied, then the probability of a collision in the next step is at least $(1-p) \wedge p = 1-p$. If the core is not occupied, then after one step it becomes occupied, and the probability of a collision on the next step is at least $1-p$. Therefore, from any configuration of particles, the probability of a collision occurring in the next two steps is always at least $1-p$, and each collision reduces $A_t$ by 2. The formula follows.
\end{proof}

Keeping in mind the identity in \thref{lem:A_t}, we will require a bound on $C_{2n}$. Because $C_t$ always has a slight drift towards zero, it can be dominated by the displacement of a simple random walk.

\begin{lemma} \thlabel{lem:C<D}
Fix $p=1/2$. Let $D'_t$ be the displacement from the origin of a simple symmetric random walk on $\mathbb{Z}$ started at $0$. There exists a coupling such that 
$$C_{t} \leq D'_{t} +1$$
for all $t \geq 0$.	
\end{lemma}

\begin{proof}
We explain how to construct $D'_t$ from $C_t$. Notice that the probability $C_t$ increases is equal to the probability of picking a particle at a leaf that is the same color as those currently occupying the core, or $1$ if no particles are there.  
%With the core occupied, the probability is strictly less than $1$, and with no occupation the probability is $1$. 
Thus, we define $D'_{t+1} = D_t' +1$ if one of the following occurs:
		\begin{enumerate}[label = (\alph*)]
			\item $D_t' = 0$,
			\item $C_t>0$ and $C_{t+1} = C_t +1$,
			\item $C_t >0$ and $C_{t+1} = C_t-1$ because of a particle moving away from the core,
			\item with probability $1/2$ if $C_t=0$ and $D_t' >0$.	
		\end{enumerate}
	Otherwise $D_{t+1}' = D_{t}'-1.$ 
	
	It is easy to check that $D_t'$ is the displacement of a simple random walk, since, when it is nonzero, it transitions up or down with equal probability (the probabilities  in (b) and (c) sum to $p=1/2$, the probability of choosing the color at the core). Moreover, $D_t'$ and $C_t$ are coupled so that $D_t'$ increases whenever $C_t$ does with one exception. The only situation in which $C_t$ can exceed $D_t'$ is if $C_{t-1}=0$ and $D_{t-1}' =1$ and $D_t' =0$. When this occurs we have $C_t = D_t' +1$. However, the gap cannot become any larger than this, because while $C_t$ is larger than $D_t'$, case (d) is prohibited, so $D_t'$ increases whenever $C_t$ does. 
\end{proof}

We will soon require an estimate on a sum via comparison to an integral. We provide the antiderivative and asymptotic behavior of that integral now. 

\begin{lemma} \thlabel{lem:integral}
It holds that $\int_1^{2n} x^{-1} \sqrt{ 2n - x} \; dx = O(\sqrt n \log n).$
\end{lemma}

\begin{proof}
This follows immediately from setting $C=2n$ in the equation
\begin{align}
\int_1^C \f{\sqrt{C-x}}x \; dx = \sqrt{C} \log C  + 2 \sqrt C \log \left(\sqrt{\frac{C-1}{C}}+1\right) -2 \sqrt{C-1}\label{eq:integral}	.
\end{align}

We obtain this formula by computing the indefinite integral $\int x^{-1} \sqrt{ C-x} \; dx.$ 
Start with the substitution $u = \sqrt{C-x}$ so that the integral becomes 
$$
-2 \int \f{u^2}{C- u^2} \;du = 2 \int \f C { u^2 - C} + 1 \; du= -2 \int \frac{1}{1-\frac{u^2}{C}} \, du+ 2 u .
$$
Next, make the substitution $s= iu/\sqrt C$ with $i = \sqrt{-1}$ so that the above is equal to 
$$
2u  -2 \int \frac{1}{1-\frac{u^2}{C}} \, du = 2u + 2 i \sqrt{C} \int \frac{1}{s^2+1} \, ds = 2u+ 2 i \sqrt{C} \tan^{-1}(s).
$$     
%\HOX{Should this be $2u-2\int{\frac{1}{1-\frac{u^2}{C}}}du$? - IC}
Substituting back $u$ and then $x$ yields
$$\int x^{-1} \sqrt{C-x} \;dx = 2 \sqrt{C-x}-2 \sqrt{C} \tanh^{-1} \left( \f{ \sqrt{C-x} }{\sqrt{C} } \right) + C_0.$$
We obtain the claimed formula at \eqref{eq:integral}  by applying the identity $$\tanh^{-1}(z) = \f 12 \left(\log(1+z) - \log(1-z) \right),$$ combining logarithmic terms, and then computing the definite integral.
\end{proof}
We now put this inequality to work in bounding $E M_{2n}$. 

\begin{lemma} \thlabel{lem:Mn}
Fix $p=1/2$. It holds that $E M_{2n} = O(\sqrt n \log n)$. 	
\end{lemma}

\begin{proof}
Let $G_t = \{ M_t = M_{t-1} +1\}$ be the event that a particle at the core is sampled at time $t$ and the particle is not annihilated after taking a step. This tracks when $M_t$ increases; accordingly, at time $2n$ we have  
%\HOX{Do we need a semicolon instead of the first coma? - IC}
$$M_{2n} = \sum_{t=1}^{2n} \ind{G_t}.$$
Notice that $P(G_t)$ is at most the probability of sampling a particle at the core at time $t$. Given $A_t= a$ and $C_t= c$, this probability is equal to $c/a$, since $p=1/2$. It follows that 
\begin{align}
P(G_t) = E [E [\ind{G_t} \mid A_t, C_t] ]\leq E[ C_t/A_t].\label{eq:G}
\end{align} 
Using this for $t <2n$ and the bound $P(G_{2n}) \leq 1$ as we bring the expectation inside the sum, we obtain 
\begin{align}E M_{2n} \leq  1+ \sum_{t=1}^{2n-1} E \left[\f{C_t}{A_t}\right] \leq 1 +\sum_{t=1}^{2n-1} \f{ E C_t }{ 2n -t}.\label{eq:M_bound} 
\end{align}
The second inequality uses the deterministic bound $A_t \geq 2n -t$ from \thref{lem:A_t}. Bounding $E C_t \le E D_t +1$ via \thref{lem:C<D} and then bounding $E D_t$ with \thref{lem:D_t}, we obtain
$$E M_{2n} \leq 1+ \sum_{t=1}^{2n-1} \f{E D_t +1}{2n -t} \leq   1+ 2\sum_{t=1}^{2n-2} \f{\sqrt{t}}{2n -t}.$$
Reindexing with $s = 2n -t$ gives
$$E M_{2n} \leq 1 + 2\sum_{s=1 }^{2n} \f{ \sqrt{ 2n - s} }{s} $$
By comparison to the integral in \thref{lem:integral}, the summation above is $O(\sqrt n \log n)$.
\end{proof}

\begin{remark} \thlabel{rem:G}
Note that in the previous argument at \eqref{eq:G} we made the bound $P(G_t) \leq E [C_t/A_t]$. One might suspect that the logarithmic factor comes from this estimate. However, the exact formula is 
$$
P(G_t) = E\left[ \f{C_t}{A_t} \f{(2n - U_t)}{2n}\right]
$$
where $U_t$ is the number of sites occupied by particles of the opposite color from the core at the leaves of $S_{2n}$ at time $t$. Exactly describing the quantity $U_t$ is subtle since it depends on clustering at the leaves and on the current particle type occupying the core. Regardless, we have $1/2 \leq (2n- U_t)/2n \leq 1$ for all $t$ since $0\leq U_t\leq n$. So, $P(G_t) \geq (1/2)E[ C_t/ A_t]$. Thus, the estimate we make on $P(G_t)$ is not the source of the logarithmic factor. \end{remark}

In any case, we now we have the necessary ingredients to prove our upper bound.

%This allows us to describe the Markov chain...
%$$P(C_{t+1} = C_{t}-1 \mid  A_t =a)= \frac{1}{2} + \frac{C_t}{a}$$
%$$P(C_{t+1} = C_{t} \mid  A_t =a)= 0$$
%$$P(C_{t+1} = C_{t}+1 \mid  A_t =a)= \frac{1}{2} - \frac{C_t}{a}$$

\begin{proof}[Proof of \thref{thm:2} (ii)]
By \thref{lem:A_t} we have
\begin{align}E A_{2n} = E C_{2n} + 2 E M_{2n}.\label{eq:A_bound}\end{align}
It follows from  \thref{lem:D_t,lem:C<D} that $E C_{2n} = O(\sqrt n)$, and from \thref{lem:Mn} we have $E M_{2n} = O(\sqrt n \log n)$. Thus, $E A_{2n} = O(\sqrt n \log n)$. Applying \thref{lem:remy_bound} with $t=2n$ gives
$$T_{1/2}^2(S_{2n}) \preceq 2n + 2\sum_{i=1}^{A_{2n}/2 } X_i(1/2).$$
It follows from Wald's lemma and our bound on $E A_{2n}$ that
$$E T_{1/2}^2(S_{2n}) \leq 2n + 4EA_{2n} = 2n + O(\sqrt n \log n).$$
\end{proof}

\section{Asymmetric speeds on the star graph} \label{sec:asymmetric}

We break this section into three subsections. The first two subsections contain technical estimates for a modified coupon collector problem, and also describe how these connect back to the two-type system. The third subsection contains the proofs of \thref{thm:p,thm:asymptotic,prop:p=1}. 

\subsection{Lemmas for the asymptotic lower bound}
The idea behind the lower bound is that after $t\approx-4\log(1-p)n$ steps approximately $r=(1-p)t$ red particles will move from their starting location. To eliminate all of the red particles by time $t$, blue particles must visit all $n-r$ of the sites with red particles that did not move. We identify the sites initially occupied by red particles as coupons, and view each jump from the core by a blue particle as an attempt to collect one of these coupons. So $T_p^2(S_{2n})$ is lower bounded by the number of steps needed for a coupon collector to collect $n - r$ coupons, which we prove has expected value on the order of $t$.

\begin{lemma} \thlabel{lem:coupon1} Fix any $\epsilon\in (0,1)$, let $t_p=t_p(n) = -n\log(1-p)$ and let $R \overset{d} = \Bin(4(1-\epsilon)t_p, 1-p)$. Consider a coupon collector process in which there are $n$ coupons. At each step, with probability $1/2$ no coupon is selected, and otherwise one is picked uniformly at random. Let $T'_p$ be the number of steps needed to sample $n-R$ distinct coupons. Then there exists $p'(\epsilon)<1$ such that for all $p>p'(\epsilon)$ we have
	$$P\left(T_p' \le 2(1-\epsilon)t_p\right) \to 0$$
	as $n\to\infty$.
\end{lemma}

\begin{proof}
The time between each new coupon discovery is a geometric random variable, so
 $$T'_p  \overset{d}= \sum_{i=0}^{n-R-1} X\left( \f{n-i}{2n}\right)$$ 
with the convention that the sum is zero if $n-R\leq 0$ and the $X$'s are independent. Let $a =5(1-\epsilon)(1-p)\log(1-p)^{-1}$, and observe that $a\to 0$ as $p\to 1$. Noting that $ER = 4an/5$, by a standard Chernoff bound for the binomial distribution, we have
\begin{equation}\label{eq:R_bound}
P(R\ge a n) \le e^{-cn}
\end{equation}
for some $c = c(\epsilon,p)$. Letting 
$$
Y = \sum_{i=0}^{n-an-1} X\left( \f{n-i}{2n}\right),
$$
then using the bound $\log m \leq  \sum_{i=1}^m i^{-1} \leq 1 + \log m$ we have
$$
EY = 2n \sum_{j=an+1}^n \frac1j \ge 2n\left[\log(1/a) -1\right] > (2-\epsilon)n\log\left(\frac1{1-p}\right),
$$
and
$$
\Var(Y) \le 4n^2 \sum_{j=an+1}^\infty \frac1{j^2}\le \frac{4n}{a}.
$$
for all $p$ sufficiently close to $1$ and $n$ large enough, depending on $p$.
Therefore, by Chebychev's inequality,
\begin{equation}\label{eq:Y_bound}
\begin{aligned}
P(Y<-2(1-\epsilon)\log(1-p) \cdot n) &\le P\left(|Y-EY| > \frac{\epsilon}{2-\epsilon} EY\right) \\
&\le \frac{4}{a\epsilon^2 (\log((1-p)^{-1}))^2} n^{-1}.
\end{aligned}
\end{equation}
Combining \eqref{eq:R_bound} and \eqref{eq:Y_bound}, we arrive at
\begin{align*}
P\left(T_p' < -2(1-\epsilon)\log(1-p) \cdot n\right) &\le P(R\ge an) + P(Y<-2(1-\epsilon)\log(1-p) \cdot n)\\
&\to 0
\end{align*}
as $n\to\infty$.
%
%
%Let $H_m = \sum_{i=1}^m i^{-1}$. 
%Using the elementary bound $\log m \leq  \sum_{i=1}^m i^{-1} \leq 1 + \log m$, we have
%\begin{align}
%E [T'_p\mid R] = 2n \sum_{i=R+1}^ni^{-1} \geq 2n \left(  \log n - 1- \log R \right).\label{eq:logg1}
%\end{align}
%By Jensen's inequality we have $E \log R \leq \log E R = \log( 2(1-p)t_p)$. 
%It follows from this and \eqref{eq:logg1} that 
%\begin{equation}
%\begin{aligned}
%E T_p' &\geq 2n [\log n - \log ER - 1] \\
%&\geq 2n \left[ \log\left( \f 1 { 1-p }\right) -  \log \log\left((1-p)^{-1}\right) - \log 2 - 1 \right]\\
%& \ge (2-\epsilon) n\log\left(\frac{1}{1-p}\right)
%\end{aligned}
%\end{equation}
%for $p$ sufficiently close to $1$.
%Plugging in the value of $t_p$ and expanding gives the right side is equal to
%\begin{align}2 n \left[ \log n  - \log(1-p) - \log\lfloor - \log(1-p) \rfloor   - \log n - \log 2 - 1 \right].\label{eq:log_expand}	
%\end{align}
%
%The $\log n$ terms cancel, and $ - \log\lfloor -\log(1-p) \rfloor  - 1 - \log 2 = o(- \log(1-p))$  as $p \uparrow 1$.  Thus, if $p$ is chosen large enough, \eqref{eq:log_expand} is larger than 
%$-(2-\epsilon)\log(1-p) n.$
%
\end{proof}

\begin{lemma} \thlabel{lem:ALB}
For any $\epsilon\in (0,1)$, there exists $p'(\epsilon)<1$ such that for each $p> p'(\epsilon)$ and all sufficiently large $n$, we have
$$ET^2_p(S_{2n}) \ge 4(1-\epsilon)^2n \log\left(\frac1{1-p}\right).$$	
\end{lemma}
%\begin{proof}
%Let $t_p$, $R$ and $T'_p$ be as in \thref{lem:coupon1}. At $t = 2t_p$, blue particles have moved to a uniformly sampled leaf at most $t_p$ times. This is because each visit to a leaf requires two steps from a blue particle.  Moreover, at time $2t_p$, at most $R$ red particles have moved from their starting locations. The maximum number of red particles that have been extinguished at time $2t_p$ are the $R$ particles that moved plus the number of distinct leaves originally containing red particles that have been visited by blue. Thus, if blue has not yet visited $n-R$ leaves originally occupied by red particles, then it is impossible that all red particles have been extinguished. As $2T_p'$ represents the minimal number of steps by blue particles for this to occur, we obtain the claimed stochastic dominance. 
%\end{proof}

%%%%% Attempt at clarifying the proof:
\begin{proof}
Let $t_p$, $R$ and $T'_p$ be as in \thref{lem:coupon1}. Let $\mathtt{Red\_Moved}$ be the number of red particles that have moved from their starting locations through time $4(1-\epsilon)t_p$, and let $\mathtt{Red\_Sites\_Visited\_by\_Blue}$ be the number of leaf vertices that were initially occupied by red particles and were visited by at least one blue particle through time $4(1-\epsilon)t_p$ (whether or not they are occupied by red at the time of blue's visit). The number of red particles extinguished through time $4(1-\epsilon)t_p$ cannot exceed
$$
\mathtt{Red\_Moved} + \mathtt{Red\_Sites\_Visited\_by\_Blue},
$$
so if all particles are to be removed by time $4(1-\epsilon)t_p$, we must have
$$
 \mathtt{Red\_Sites\_Visited\_by\_Blue} \ge n-\mathtt{Red\_Moved}.
$$

Note that $\mathtt{Red\_Moved}$ cannot exceed the number of times that red particles are chosen to move through time $4(1-\epsilon)t_p$, which has the same distribution as $R$, so we have $\mathtt{Red\_Moved}\preceq R$.
Also, through time $4(1-\epsilon)t_p$, a blue particle has moved to a uniformly sampled leaf on at most $2(1-\epsilon)t_p$ steps, since each visit to a leaf requires two moves by a blue particle. The random variable $\mathtt{Red\_Sites\_Visited\_by\_Blue}$ is therefore stochastically dominated by the number of distinct coupons collected after $2(1-\epsilon)t_p$ steps, where at each step, with probability $1/2$ no coupon is selected, and otherwise one of $n$ coupons is selected uniformly at random. This number can be taken to be independent of $R$, as the number of steps taken by the coupon collector is deterministic, so we are in the setting of \thref{lem:coupon1}, and we have
$$
P(T_p^2(S_{2n}) \le 4(1-\epsilon)t_p) \le P(T_p' \le 2(1-\epsilon)t_p) \to 0
$$
as $n\to\infty$. Letting $n$ be sufficiently large so that the probability above is smaller than $\epsilon$ gives the desired lower bound on the expectation.
\end{proof}
%%%%%%

\subsection{Lemmas for the asymptotic upper bound}
The idea behind the upper bound is to run the process for $t=- 8n \log(1-p)$ steps. At this point, we prove that blue has moved to nearly $pn$ of the sites that were initially red, and at most $r=(1-p)t = (1-p)\log((1-p)^{-8}) n$ red particles have moved to avoid a collision. This means that at most $n -pn + r$ red particles have avoided collision through time $t$. We then use the bound at \thref{lem:remy_bound} to show that the expected time to destroy the remaining particles is $O(n)$ with leading constant that does not depend on $p$. 

\begin{lemma} \thlabel{lem:coupon2}
Consider the following coupon collection process like that of \thref{lem:coupon1}. Let $t_p =  -n\log(1-p)$ and for $r>4$, let $B \overset{d}= \Bin(rt_p , p)$. Let $N = (B - n)/2$, and let $V$ be the number of the $n$ coupons that are \emph{not} collected through $N$ steps, where we set $V=n$ if $N\le 0$. For all fixed $p$ sufficiently close to $1$ and for all $n$ sufficiently large,
$$P(V \ge (1-p)n) \le 3/n.$$
\end{lemma}

\begin{proof}
%Conditional on $N$, the probability of a fixed coupon not being selected through $N$ steps is $(1- (2n)^{-1})^N \1_{\{N>0\}} \le (1- (2n)^{-1})^N$, so
%\begin{align}
%E V \le nE\left(1- \f 1 {2n} \right)^N.\label{eq:V}
%\end{align}
Note that $B$ is a binomial random variable with mean $rp t_p$, so for $p$ sufficiently close to $1$, standard large deviation estimates for the binomial distribution imply that
\begin{align}\label{eq:binomial}
P\left(B\le rpt_p\left(1-(rpt_p)^{-1/4}\right)\right) \le e^{-(rpt_p)^{1/2}/16} \le e^{-n^{1/2}}
\end{align}
for all large enough $n$. Let $a = a(r,p) := \frac12 (rp\log\frac{1}{1-p} - 2)$, which is large for $p$ close to $1$, so for all large enough $n$, we have
\begin{align*}
P(N \le an) & \le P\left(N\le \frac12[rpt_p(1-(rpt_p)^{-1/4})-n]\right)\\
& = P\left(B\le rpt_p\left(1-(rpt_p)^{-1/4}\right)\right)\\
&\le e^{-n^{1/2}}.
\end{align*}
Letting $V'$ be the number of coupons \emph{not} collected through $an$ steps, we have
\begin{align}
EV'  &= n(1-1/2n)^{an}\\
&\le ne^{-a/2} \\
&\le n\exp\left[-\frac14rp\log\frac{1}{1-p}  + 1\right]\\
&\le n\frac12(1-p),
\end{align}
where in the last line we use the assumption $r>4$ and take $p$ such that $rp/4 > 1$.

%\begin{align}EV &\le EV'  + nP(N \leq an)	\\
%& \le n(1-1/2n)^{an} + ne^{-n^{1/2}}\\
%&\le ne^{-a/2} + ne^{-n^{1/2}}\\
%&\le n\exp\left[-\frac14rp\log\frac{1}{1-p}  + 2\right]\\
%&\le n\frac12(1-p),
%\end{align}

In anticipation of our variance bound, observe that for $n$ sufficiently large, by Taylor's theorem, we have
\begin{align*}
(1-1/n)^{an} - (1-1/2n)^{2an} &\le e^{-a}(e^{a/2n} - e^{a/8n}) \\
&\le e^{-a}(a/n - a/8n)\\ 
&\le ae^{-a}/n.
\end{align*}
The probability that coupons labeled $1$ and $2$ (say) are not chosen through $an$ steps is $(1-2/2n)^{an}$, so we have for $p$ close to $1$,
\begin{align*}
\Var(V') &\le EV' + n^2\left[(1-1/n)^{an} - (1-1/2n)^{2an} \right]\\
&\le n(e^{-a/2} + ae^{-a}) \\
%&\le 2e^{-a/2} n\\
&\le (1-p)n.
\end{align*}
Finally, noting that $V \le V'\1_{\{N>an\}} + n\1_{\{N\le an\}}$, we have
\begin{align*}
P(V \ge (1-p)n) &\le P(V' \ge (1-p)n) + P(N\le an)\\
&\le P(|V' - EV'| \ge (1-p)n/2) + e^{-n^{1/2}}\\
&\le 3/n.
\end{align*}
\end{proof}

\begin{lemma} \thlabel{lem:AUB}
For any $\epsilon>0$ there exists $p'(\epsilon)<1$ such that for each $p>p'(\epsilon)$ and all sufficiently large $n$, we have
$$
ET_p^2 \le (12+\epsilon) n \log\frac{1}{1-p}.
$$
% and thus by \thref{lem:remy_bound}
%$$T_2^p(S_{2n}) \preceq 6 t_p + \sum_{i=1}^{n - V} X(1-p).$$
\end{lemma}

\begin{proof}
Let $r>4$ and $t_p$ be as in \thref{lem:coupon2}. Let $B \overset{d}= \Bin(rt_p , p)$ be the number of times that blue is chosen to move through time $t=rt_p$ (based on the values of $(Z_s)_{s\le t}$ defined at the start of Section~\ref{sec:2S}). Like in the proof of \thref{lem:ALB}, we let $\mathtt{Red\_Moved}$ be the number of red particles that have moved from their starting locations through time $t$, and let $\mathtt{Red\_Sites\_Visited\_by\_Blue}$ be the number of leaf vertices that were initially occupied by red particles and were visited by at least one blue particle through time $t$. 

The number of times that blue particles are chosen to move from the core to a leaf through time $t$ must be at least $N=(B-n)/2$, provided particles persist through time $t$. To see this, let $C^{\rightarrow}$ be the number of jumps from the core to leaves up to time $t$, and $C^{\leftarrow}$ be the number of jumps from leaves to the core up to time $t$. Adding the equations
\begin{align}
	C^{\rightarrow} + C^{\leftarrow} &= B\\\
	C^{\rightarrow} - C^{\leftarrow} &\geq -n,	
\end{align}
and solving for $C^\rightarrow$ gives the claimed inequality $C^\rightarrow \geq N$.
We now satisfy the hypotheses of \thref{lem:coupon2}. Therefore,
$$
P(\mathtt{Red\_Sites\_Visited\_by\_Blue}\le pn, A_t>0) \le P(V\ge (1-p)n) \le 3/n.
$$
Moreover, by \eqref{eq:binomial}, for $p$ close to $1$ and sufficiently large $n$ we have
\begin{align*}
P\left(\mathtt{Red\_Moved}\ge nr(1-p)\left(\log\frac{1}{1-p} +1\right)\right) &\le P\left(B\le rpt_p\left(1-(rpt_p)^{-1/4}\right)\right)\\
&\le e^{-n^{1/2}}.
\end{align*}

Observe that $A_t \le 2(n - \mathtt{Red\_Sites\_Visited\_by\_Blue}+\mathtt{Red\_Moved})$, since red particles that have not moved and are at sites that are visited by blue particles by time $t$ must be eliminated. Combining this observation with the last two inequalities gives
\begin{align*}
&P\left(A_t \ge 2n \left((1-p) + r(1-p)\left(\log\frac{1}{1-p}+1\right)\right) \right) \\
& \qquad \qquad \le P(\mathtt{Red\_Sites\_Visited\_by\_Blue}\le pn, A_t>0) \\
&\qquad \qquad \quad+ P\left(\mathtt{Red\_Moved}\ge nr(1-p)\left(\log\frac{1}{1-p} +1\right)\right) \\
&\qquad \qquad \le 3/n + e^{-n^{1/2}}\\
&\qquad \qquad \le 4/n.
\end{align*}
Since we have $A_t\le 2n$, we arrive at
$$
EA_t \le 2n \left((1-p) + r(1-p)\left(\log\frac{1}{1-p}+1\right)\right) + 2n(4/n).
$$
Combining this bound with \thref{lem:remy_bound} applied at time $t=rt_p$ and Wald's equation gives
\begin{align*}
ET_2^p &\le nr\log\frac{1}{1-p} + 2E[A_t/2] \frac{1}{1-p}\\
&\le n\left[3r\log\frac{1}{1-p} + 2 + 2r\right] + \frac{1}{2(1-p)}.
\end{align*}
Taking $r$ close to $4$, then $p$ close enough to $1$, then $n$ sufficiently large completes the proof.
%
%\HOX{Following Lemma 9, maybe it works for $N = (B - R -n)/2$? This seems like a combinatorial feature of tracing the extended sampling process. -MJ} \HOX{I think I agree, so this argument will only succeed for sufficiently large $p$. -DS}
%At time $8 t_p$, either there are no red particles remaining, or blue particles have been sampled to move $B$ times with $B=8t_p - R$ in distribution as in \thref{lem:coupon2}. The same reasoning as in \thref{lem:M_t} requires that a blue particle moves from the core to a leaf at least $N=(B-n-1)/2$ times. Also at time $8t_p$, at most an $R$-distributed number of red particles have moved from their starting location. As the jumps from the core are to uniformly sampled leaves, a $V$-distributed number of these sites are visited. It maximizes the number of particles in the system if we assume $R$ of the sites counted by $V$ were inhabited by red particles that moved before being visited. In this worst-case setting, we have $V-R$ collisions must occur. Thus, there are at most $2n - 2(V-R)$ particles remaining in the system at time $8 t_p$.
%
\end{proof}

\subsection{Proofs}

\begin{proof}[Proof of \thref{thm:p}]
The upper bound follows from \thref{lem:remy_bound} by setting $t=0$ and taking expectation. 
The lower bound follows from similar reasoning as the proof of \thref{thm:2} (i).  \thref{lem:A_t,lem:TLB,lem:M_t} together imply that 
\begin{align}
ET^2_p(S_{2n}) \geq 2n + (1/4)E D_t -1.\label{eq:last}	
\end{align}
Recalling the definition $D_t = | \sum_1^t Z_s|$ at the start of this section, we have 
$$D_t \geq Z_1 + \cdots + Z_t.$$
Since $E Z_1 = 2p -1$, we then have 
$$E D_{2n} \geq (2p-1)2n.$$
Applying this inequality at \eqref{eq:last}, it follows that 
$$E T^2_p(S_{2n}) \geq 2n + \f{2p-1}{2} n -1.$$
\end{proof}

\begin{proof}[Proof of \thref{thm:asymptotic}]
The lower bound follows from \thref{lem:ALB} and the upper bound from \thref{lem:AUB}.
\end{proof}

\begin{proof}[Proof of \thref{prop:p=1} ]
Let $M$ denote the value of $M_t$ at time $t= T_1^2(S_{2n})$. Evaluating \thref{lem:A_t} at $t=T^2_1(S_{2n})$ and rearranging gives
\begin{align}
T_1^2(S_{2n}) = 2n + 2 M.	\label{eq:TM}
\end{align}
Let $M(i)$ be the number of times a particle is sampled at the core and moves without being annihilated when there are $2i$ particles in the system. Since red particles do not move, each time a particle at the core is sampled there is an $i/2n$ chance of annihilation, and annihilations cannot occur in any other way. It follows that $M(i)$ has distribution $X(i/2n)-1$ and thus
$$M \overset{d}= \sum_1^{n} (X(i/2n)-1).$$
In light of \eqref{eq:TM} and the above equality, we have
$$T_1^2(S_{2n}) \overset{d}= 2n + 2 \sum_1^{n} (X(i/2n)-1) = 2 \sum_1^n X(i/2n),$$
which has expectation $4 n \log n + 4 \gamma n + o(n)$.
\end{proof}

\subsection*{Acknowledgements} Thanks to Rick Durrett for helpful conversations, encouragement, and advice.

\bibliographystyle{amsalpha}
\bibliography{parking}

\end{document}